\documentclass[10pt]{article}

\usepackage{amsmath,amssymb,amsthm}
\usepackage{enumitem}
\usepackage{mathtools}
\usepackage[left=4cm,right=3cm,top=3cm,bottom=3cm]{geometry}
\usepackage{listings}
\usepackage{hyperref}
\usepackage{amsrefs}

\newcommand{\NN}{\ensuremath{\mathbb N}}

\newcommand{\cA}{\mathcal{A}}
\newcommand{\cG}{\mathcal{G}}

\newcommand{\cV}{\mathcal{V}}
\newcommand{\cW}{\mathcal{W}}
\newcommand{\cB}{\mathcal{B}}

\newcommand{\floor}[1]{\mbox{$\left\lfloor #1 \right\rfloor$}}

\newtheorem{thm}{Theorem}
\newtheorem{lem}[thm]{Lemma}

\newtheorem{claim}[thm]{Claim}
\newtheorem{cor}[thm]{Corollary}

\begin{document}
  \title{On the Thickness of Infinite Generalized Sidon Sets, I}
  \author{Kevin O'Bryant\thanks{Email: \texttt{kevin.obryant@csi.cuny.edu}.}}
  \date{\today}
  \maketitle

\begin{abstract}
  Let $g \ge1$. A set $\cA$ of nonnegative integers is a Sidon set if for each $d>0$ there is at most one pair $(a,b)\in\cA\times\cA$ with $d=a-b$. If there are at most $g$ pairs, then $\cA$ is a $g$-Golomb ruler. We prove that if $\cA$ is a $g$-Golomb ruler, then
  \[\liminf_{n\to\infty} \frac{|\cA\cap[0,n)|}{\sqrt{n/\log n}} \le \frac{2\sqrt g }{\sqrt{\log 2}},\]
  generalizing and sharpening results of Erdős and Cilleruelo.
  There is a $g$-Golomb ruler $\cG$ with
  \[\frac{\sqrt g }{\sqrt2} \le \limsup_{n\to\infty} \frac{|\cG\cap[0,n)|}{\sqrt n} \le \sqrt{g } ,\]
  generalizing a result of Kr\"uckeberg.
\end{abstract}

\section{Introduction}
Let $g \ge1$. A set $\cA$ of nonnegative integers is a $g$-Golomb ruler if each $d>0$ has at most $g$ pairs $(a,b)\in\cA\times\cA$ with $d=a-b$. With $g =1$, these are known as Sidon sets in combinatorics, Golomb rulers in recreational math, Babcock sets in electrical engineering, and $B_2$ sequences in number theory. See the author's comprehensive annotated bibliography~\cite{2004.Obryant} for hundreds of citations.

Our main result is giving the constant in the following theorem. Erd\H{o}s~\cite{1955.Stohr} proved finiteness for Sidon sets, and Cilleruelo~\cite{2015.Cilleruelo-a} proved $8\sqrt{7} \approx 21.2$, and suggests bringing the constant down to $4$ as an exercise. We bring the constant down to $2/\sqrt{\log2} \approx 2.4$. The extension from Sidon sets to $g$-Golomb rulers is new, but easy.
\begin{thm}\label{thm:liminf}
    Let $g$ be a positive integer. 
    If $\cA$ is a $g$-Golomb ruler, then
    \[\liminf_{n\to \infty} \frac{A(n)}{\sqrt{n/\log(n)}} \le \frac{2}{\sqrt{\log 2}} \,\sqrt g .\]
\end{thm} 
\begin{cor}
  Let $g$ be a positive integer, and $\cA=\{0\le a_1<a_2<\dots\}$ an infinite $g$-Golomb ruler. Then
    \[\limsup_{n\to \infty} \frac{a_n}{n^2\log n} \ge \frac{\log 2}{2g}.
    \]
\end{cor}
There is a counterpart to Theorem~\ref{thm:liminf} that also originates with Erd\H{o}s~\cite{1955.Stohr}, and was improved by Kr\"uckeberg~\cite{1961.Kruckeberg}. We extend Kr\"uckeberg's result to $g$-Golomb rulers in Theorem~\ref{thm:limsup}.
\begin{thm}\label{thm:limsup}
    Every $g$-Golomb ruler $\cA$ has
    \[\limsup_{n\to\infty} \frac{|\cA \cap [0,n)|}{\sqrt{n}} \le \sqrt{g } .\]
    There is a $g$-Golomb ruler $\cG$ with 
    \[\limsup_{n\to\infty} \frac{|\cG \cap [0,n)|}{\sqrt{n}} \ge \frac{1}{\sqrt2}\,\sqrt g .\]
\end{thm}    

Naturally, one asks if these results are close to best possible. Ruzsa~\cite{1998.Ruzsa-a} gives a beautiful construction of a Sidon set with $A(n) = n^{-1+\sqrt2+o(1)}$, and this remains the record for an infinite set. Lemmas~\ref{lem:finite upper} and~\ref{lem:finite lower} below give upper and lower bounds on the size of $g$-Golomb rulers in $[0,N)$.

This is Part I of a series of 3 works by the author. In Part I, we handle infinite $g$-Golomb rulers. In Part II~\cite{2026.Obryant-b}, we handle infinite $B_h$-sets with $h$ even. In Part III~\cite{2026.Obryant-c}, we address infinite $B_h$-sets with odd $h$.

In the next section, we reproduce two results from~\cite{2015.Caicedo&Martos&Trujillo} that we need for this work. In Section~\ref{sec:Proof of liminf}, we give the proof of Theorem~\ref{thm:liminf}. In Section~\ref{sec:Proof of limsup}, we give the proof of Theorem~\ref{thm:limsup}. In Section~\ref{sec:problems}, we list natural open problems suggested by this work.

\section{Literature}\label{sec:literature}
We quote two results from Caicedo \& Martos \& Trujillo~\cite{2015.Caicedo&Martos&Trujillo}.
\begin{lem}[Caicedo \& Martos \& Trujillo~\cite{2015.Caicedo&Martos&Trujillo}]
  \label{lem:finite upper}
  Let $g ,N$ be positive integers. If $\cA$ is a $g$-Golomb ruler with $\max\cA-\min\cA < N$, then 
    \[|\cA| \le (g N)^{1/2}+(g N)^{1/4}+1\le 3\sqrt{g N}.\]
\end{lem}
The crude ``$3\sqrt{g N}$'' bound is wasteful, but enough for some of our uses below. Sometimes, we will use Lemma~\ref{lem:finite upper} in the inexplicit form $|\cA| \le \sqrt{g N} + o(\sqrt N)$.

\begin{lem}[Caicedo \& Martos \& Trujillo~\cite{2015.Caicedo&Martos&Trujillo}]
    \label{lem:finite lower}
    Let $g \ge 1$ be an integer. If $q$ is a prime power with $q\equiv 1 \pmod{g }$, then there exists a $g$-Golomb ruler $\cB\subseteq[0,(q^2-1)/g )$ with $|\cB|=q$.
\end{lem}

\section{The Proof of Theorem~\ref{thm:liminf}}\label{sec:Proof of liminf}
The general structure of the proof is the same as Erd\H{o}s's for Sidon sets: we break the $g$-Golomb ruler $\cA$ into blocks of length $N$, and will eventually take $N\to\infty$. We consider the energy
    \[ E \coloneqq \sum_\ell \left| \cA \cap [(\ell-1)N,\ell N) \right|^2\]
and get an upper bound from the $g$-Golomb property and a lower bound from Cauchy's Inequality.

We believe that we have fully optimized this argument. Nevertheless, after the proof, we discuss some alternative approaches and make some guesses about where improvements could originate. We apply Lemma~\ref{lem:energy} in this work with $\cB=\cA$ a $g$-Golomb ruler and $c=g/2$. In Part II of this series, we will use Lemma~\ref{lem:energy} with $\cB$ the $k$-fold sum of a $B_{2k}$-set and $c=1/2$. There is a similar argument in Part III, but there $\sqrt{n/\log n}$ is replaced with $\sqrt{n}$.

\begin{lem}\label{lem:energy}
Let $\cB\subseteq\NN$ have counting function $B(n)$.
Let $M=M(N)$ satisfy, as $N\to\infty$,
\begin{enumerate}[label=\textit{(\roman*)},noitemsep]
  \item\label{eng:window} $B((M+1)N)=o(N)$;
  \item\label{eng:ratio} $\log M/\log N =1+o( 1 )$;
  \item\label{eng:head} $B(N)=o(\sqrt{N \log N})$.
\end{enumerate}
Suppose there is a constant $c$ such that there is an
offset $t^\ast=t^\ast(N)\in[0,N)$ whose block counts
  \begin{equation*}\label{hyp:blocks}
    F_\ell \coloneqq B(t^\ast+\ell N)- B(t^\ast+(\ell-1)N)
  \end{equation*}
satisfy (as $N\to\infty$)
  \begin{equation}\label{hyp:energy}
    \sum_{\ell=1}^{M}\binom{F_\ell}{2}\le cN+o(N).
  \end{equation}
Then
  \[\liminf_{m\to\infty}\frac{B(m)}{\sqrt{m/\log m}}\le \sqrt{\frac{8c}{\log 2}}.\]
\end{lem}

\begin{proof}
Assume $N\ge 3$, and set $\tau_N,\widetilde{\tau}_N$ as
  \begin{equation}\label{def tau}
    \tau_N \coloneqq \inf_{n\ge N} B(n)\sqrt{\frac{\log n}{n}},
    \qquad 
    \widetilde{\tau}_N \coloneqq \min\{\tau_N , 1+\sqrt{\frac{8c}{\log 2}}\},
  \end{equation}
so that $\tau_N$ is nondecreasing, $\lim_N\tau_N=\liminf_m \frac{B(m)}{\sqrt{m/\log m}}$, and $\widetilde{\tau}_N$ is bounded.
As $\tau_N$ is monotone, it suffices to show $\tau_N\le
\sqrt{8c/\log2}$ for arbitrarily large $N$. 
Given $N$, take $t^\ast$ as provided, and consider the ``energy''
  \[E\coloneqq \sum_{\ell=1}^{M}F_\ell^2
      =2\sum_{\ell=1}^{M}\binom{F_\ell}{2}+\sum_{\ell=1}^{M}F_\ell.\]

The hypothesis~\eqref{hyp:energy} bounds the first sum by $2cN+o(N)$, and the second sum telescopes:
  \[\sum_{\ell=1}^{M}F_\ell=B(t^\ast+\floor{M}N)-B(t^\ast)\le B((M+1)N) = o(N)\]
by~\ref{eng:window}. Hence 
\begin{equation}\label{eq:upper}
  E\le 2cN+o(N).
\end{equation}

The lower bound on $E$ uses Cauchy's Inequality with the weights
\[
   w_\ell\coloneqq \begin{cases}
    \left( \ell\,\log \ell N \right)^{-1/2} & 1\le\ell\le M; \\
    0 & \text{otherwise.}
   \end{cases}
\]
By Cauchy's inequality,
\begin{equation}\label{eq:CS}
   E \coloneqq \sum_{\ell=1}^{M}F_\ell^{2}
   \ge \frac{\bigl(\sum_{\ell=1}^{M}w_\ell F_\ell\bigr)^{2}}
           {\sum_{\ell=1}^{M}w_\ell^{2}} .
\end{equation}
We need an upper bound on the denominator ${\sum_{\ell=1}^{M}w_\ell^{2}}$, and a lower bound on the numerator $\bigl(\sum_{\ell=1}^{M}w_\ell F_\ell\bigr)^{2}$.

\begin{claim}\label{eq:den}
    \(\displaystyle {\sum_{\ell=1}^{M}w_\ell^{2}} \le \log 2 + o(1)\).
\end{claim}
\begin{proof}[Proof of Claim~\ref{eq:den}]
As $x\mapsto\bigl(x\,\log xN \bigr)^{-1}$ is decreasing,
\begin{align*}
   \sum_{\ell=1}^{M}w_\ell^{2}
   &=\sum_{\ell=1}^{M}\frac{1}{\ell\,\log \ell N}\\
   &\le\frac{1}{\log N}+\int_{1}^{M}\frac{dx}{x\,\log xN}\\
   &=\frac{1}{\log N}+\log\frac{\log MN}{\log N}.
\end{align*}
Now, by~\ref{eng:ratio}, the ratio $\log(MN)/\log(N)\to 2$, so that 
  \(   \sum_{\ell=1}^{M}w_\ell^{2}\le o(1)+\log 2.\)
\end{proof}

\begin{claim}\label{eq:num}
    \(\displaystyle \bigg(\sum_{\ell=1}^{M}w_\ell F_\ell\bigg)^{2} \ge \widetilde{\tau}_N^2\left(\frac{\log 2}{2}\right)^2 N+ o(N).\)
\end{claim}
\begin{proof}[Proof of Claim~\ref{eq:num}]
To ease the notation visually, set $\beta_i=B(T+iN)$, and note that $F_\ell = \beta_\ell - \beta_{\ell-1}$. Also, set $M'=\floor{M}$. Rearranging the summation gives
\begin{align*}
   \sum_{\ell=1}^{M'}w_\ell F_\ell
   &=\sum_{\ell=1}^{M'}w_\ell(\beta_\ell-\beta_{\ell-1})\\
   &=\sum_{\ell=1}^{M'-1}\beta_\ell\,(w_\ell-w_{\ell+1})+w_{M'} \beta_{M'}-w_1\beta_0 \\
   &\ge\sum_{\ell=1}^{M'-1}\beta_\ell\,(w_\ell-w_{\ell+1})-w_1\beta_0.
\end{align*}
Now, we have from~\ref{eng:head}
    \[w_1\beta_0=\frac{B(T)}{\sqrt{\log N}}\le \frac{B(N)}{\sqrt{\log n}} = \frac{o(\sqrt{N \log N})}{\sqrt{\log N}} =o(\sqrt N)\]
For $1\le\ell\le M'-1$ we have 
$T+\ell N\ge\ell N\ge N$, so the definitions of $\tau_N,\widetilde{\tau}_N$ on Line \eqref{def tau} yields
  \[\beta_\ell\ge\tau_N \sqrt{\ell N/\log \ell N}  \ge \widetilde{\tau}_N \sqrt{\ell N/\log \ell N} .\] 
Writing 
    \[v_\ell\coloneqq \sqrt{\ell N/\log \ell N}\] 
and using $w_\ell-w_{\ell+1}\ge0$,
\[
   \sum_{\ell=1}^{M'}w_\ell F_\ell
   \ge\widetilde{\tau}_N\sum_{\ell=1}^{M'-1}v_\ell(w_\ell-w_{\ell+1})-o(\sqrt N) .
\]
By rearranging the summation again, with $v_0\coloneqq 0$,
\[
   \sum_{\ell=1}^{M'-1}v_\ell(w_\ell-w_{\ell+1})
   =\sum_{\ell=1}^{M'-1}w_\ell(v_\ell-v_{\ell-1})-w_{{M'}}v_{M' -1}.
\]
Since the $v$ sequence is positive and strictly increasing and the $w$ sequence is positive, we have
    \begin{equation*}
    0 < w_{M'}v_{M'-1} 
    \le w_{M'}v_{M'}
    = \frac{\sqrt{N}}{\log M'N} 
    =o(\sqrt N).
    \end{equation*}

At this point, we have
    \[\sum_{\ell=1}^{M} w_\ell F_\ell \ge o(\sqrt N)+ \widetilde{\tau}_N \sum_{\ell=1}^{M'-1} w_\ell(v_\ell-v_{\ell-1}),\]
where $w_\ell$ and $v_\ell$ are explicit sequences with easily handled properties. Set $W(x)=(x \log xN)^{-1/2}$ and $V(x)=(xN/\log xN)^{1/2}$. For $N\ge e^2$, both $W$ and $V'$ are positive and decreasing on $x\ge 1$. 
For $2\le \ell\le M'-1$, since $V'$ is decreasing and $W(x)\le W(\ell)=w_\ell$ for $x\ge\ell$,
  \[w_\ell(v_\ell - v_{\ell-1}) = w_\ell\int_{\ell-1}^{\ell} V'(x)\,dx
    \ge w_\ell\int_{\ell}^{\ell+1} V'(x)\,dx
    \ge \int_{\ell}^{\ell+1} W(x)V'(x)\,dx.\]
For $\ell=1$ (recalling $v_0=0$), we have $V(2)\le\sqrt2\,V(1)\le 2V(1)$, so that
  \[w_1(v_1-v_0) = W(1)V(1) \ge W(1)\bigl(V(2)-V(1)\bigr) \ge \int_1^2 W(x)V'(x)\,dx.\]
Summing over $1\le\ell\le M'-1$,
\begin{align*}
    \sum_{\ell=1}^{M'-1} w_\ell(v_\ell-v_{\ell-1})
    &\ge \int_1^{M'} W(x) V'(x)\,dx \\
    &=\frac{\sqrt{N}}{2} \left(\log\left(1+\frac{\log M'}{\log N}\right)- \frac{\log M'}{\log(N) \log M'N}\right)
\end{align*}
By~\ref{eng:ratio}, we have $\log M' /\log N = 1+o(1)$ and so
\begin{align*}
  \sum_{\ell=1}^{M'} w_\ell(v_\ell-v_{\ell-1})
    &=\frac{\sqrt{N}}{2} \left(\log(2+o(1))-o(1)\right) \\
    &=\frac{\log 2}{2}\sqrt{N} + o(\sqrt N).
\end{align*}
Thus,
    \begin{align*}
    \left(\sum_{\ell=1}^{M} w_\ell F_\ell\right)^2
    &\ge \left( o(\sqrt{N})+ \widetilde{\tau}_N \left(\frac{\log 2}{2}\sqrt{N} + o(\sqrt N)\right)\right)^2 \\
    &= \widetilde{\tau}_N^2\frac{\log^2 2}{4} N + o(N).
    \end{align*}
Claim~\ref{eq:num} is verified.
\end{proof}

Using Claim~\ref{eq:den} and Claim~\ref{eq:num} in Cauchy's Inequality~\eqref{eq:CS}, we get a lower bound on the energy:
\begin{equation}\label{eq:lower}
   E\ge\frac{\widetilde{\tau}_N^{2}\bigl(\tfrac{\log2}{2}\bigr)^{2}N\,(1+o(1))}
            {(\log2)(1+o(1))}
   =\widetilde{\tau}_N^{2}\,\frac{\log2}{4}\,N+o(N).
\end{equation}
Combine the upper bound on Line~\eqref{eq:upper} and lower bound on Line~\eqref{eq:lower}, take $N\to\infty$, and we have
    \[ \left(\lim_{N\to\infty} \widetilde{\tau}_N\right)^2\frac{\log 2}{4} \le 2c,\]
where the limit exists because $\widetilde{\tau}_N$ is nondecreasing and bounded.
This implies that $\widetilde{\tau}_N\le \sqrt{\frac{8c}{\log 2}}$, so that by definition $\widetilde{\tau}_N=\tau_N$.
The inequality $\tau_N \le \sqrt{\frac{8c}{\log 2}}$ is the conclusion of Lemma~\ref{lem:energy}.
\end{proof}

We are now positioned to prove Theorem~\ref{thm:liminf} quickly.

\begin{proof}
Let $g$ be a positive integer, $\cA$ a $g$-Golomb ruler with counting function $A(n) \coloneqq |\cA \cap[0,n)|$. Set $M \coloneqq N/\log N$. We will appeal to Lemma~\ref{lem:energy} with $\cB=\cA$, and we need to prove that the hypotheses of Lemma~\ref{lem:energy} are satisfied. Using Lemma~\ref{lem:finite upper}, we have
\begin{enumerate}[label=\textit{(\roman*)},noitemsep]
  \item $A((M+1)N) = A(N^2/\log N) \le 3 \sqrt{gN^2/\log N} = o(N)$;
  \item $\frac{\log M}{\log N} = \frac{\log N- \log\log N}{\log N}= 1-o(1)$;
  \item $A(N)\le 3 \sqrt{gN} = o(\sqrt{N \log N})$.
\end{enumerate}

Now, we consider 
    \[F_\ell^{(t)}\coloneqq \left| \cA\cap[\,t+(\ell-1)N,\ t+\ell N)\right|,\] 
where $0\le t < N$ and $1\le \ell \le M$.
Any particular unordered pair $a>b$ in $\binom{\cA}{2}$ with difference $d=a-b$ (and where $1\le d < N$) is usually counted in
$F_\ell^{(t)}$ for $N-d$ different values of $t$, but will lie in fewer if $a<N$, so that
\begin{align}
   \sum_{t=0}^{N-1}\sum_{\ell=1}^{M}\binom{F_\ell^{(t)}}{2}
   &\le \sum_{t=0}^{N-1}\sum_{\ell=1}^{\infty}\binom{F_\ell^{(t)}}{2} \notag \\
   &\le \sum_{\substack{\{a,b\}\subseteq\cA\\1\le a-b < N}} \bigl(N-(a-b)\bigr) \notag\\
   &\le g \sum_{d=1}^{N-1}(N-d) \notag\\
   &=\frac{g}{2} N - \frac g2. \label{eq:half}
\end{align}
Fix an offset $t^\ast$ attaining at most the average, and define
    \[F_\ell\coloneqq F_\ell^{(t^\ast)}.\]

The hypotheses of Lemma~\ref{lem:energy} are satisfied with $c=g/2$, and the conlcusion of that lemma is the conclusion of Theorem~\ref{thm:liminf}.
\end{proof}

\subsection{Nonrigorous thoughts about the proof}
The author has given each of the suggestions below serious thought and effort and has received no benefit from them, but is not convinced that no benefit is possible.

One only needs to take $N$ through a subsequence to $\infty$, and this freedom plays no role in the proof. Perhaps some $N$ allow for an improvement to Inequality~\eqref{eq:half}; perhaps it is possible to also average over some values of $N$. 

The upper bound in Inequality~\eqref{eq:half} is unchanged if one sums $\ell$ to $\infty$. That is, $M$ plays no role here. This suggests that there is a cleaner way to handle the infinite set $\cA$ instead of as a series of truncations $\cA \cap [0,MN) = \cA \cap [0,N^2/\log N)$. To this end, note that if we define energy as
    \[E_b \coloneqq \sum_{\ell=1}^\infty \binom{F_\ell}{2},\]
then $E_b$ is finite, and in fact the upper bound on $E_b$ is $g (N-1)/2$, and is a few lines easier to prove than the upper bound on $E$. The lower bound, however, seems not to benefit at all and that part of the argument demands some truncation anyway.

The transition to $F_\ell$ in the argument is (up to normalization) that of taking the conditional expectation of the indicator function of $\cA$ relative to the $\sigma$-algebra generated by $\{[T+iN,T+(i+1)N): i\ge 0\}$. Perhaps there is some reverse-martingale behind the scenes, and the current work is merely the first step of that martingale.

Related to the last suggestion, there are powerful entropy inequalities that may be relevant, with entropy taking the role played by energy.

The use of Cauchy's Inequality~\eqref{eq:CS} is optimal only if $F_\ell \approx c\cdot w_\ell$ for some constant $c$. The author knows no reason why $F_\ell$ would decay this smoothly. For example, if $F_\ell$ decreases consistently, then the set
    \[\widetilde{\cA}_N \coloneqq (T+MN-\cA)\cap[0,MN)\]
is a $g$-Golomb ruler contained in $[0,MN)$ whose $F_\ell$ sequence is consistently increasing, and $A(T+MN)-A(T)=\widetilde{A}_N(MN)$. However, the $\widetilde{\cA}_N$ set does not have the same lower bound on its infimum, and so it is unclear how to use this to advantage. Perhaps assuming that $F_\ell$ does decrease smoothly allows one to increase $M$ beneficially.

\section{Proof of Theorem~\ref{thm:limsup}}\label{sec:Proof of limsup}

This proof closely follows that given in Halberstam \& Roth~\cite{1966.Halberstam&Roth}, modified to allow $g >1$.

The first sentence of Theorem~\ref{thm:limsup} follows immediately from Lemma~\ref{lem:finite upper}. For the second sentence, we need to construct an infinite $g$-Golomb ruler, which we will do by taking a union of finite rulers, discarding a negligible number of elements at each stage. The next lemma addresses the basic situation of combining two sets.

\begin{lem}\label{lem:splice}
  Let $g \ge 1$, and let $V_2,W_1,m$ be three positive integers with 
  \begin{equation*}
    W_1 - V_2 \ge \max \{V_2, m\}.
  \end{equation*}
  Let $\cV,\cW$ be $g$-Golomb rulers contained in $[0, V_2)$, $[W_1, W_1+m)$, respectively.
  Then there is a subset $\cW^{*} \subseteq \cW$ with
  \begin{equation*}
    |\cW^{*}| \ge |\cW| - g \,\binom{|\cV|}{2}
  \end{equation*}
  such that $\cV \cup \cW^{*}$ is a $g$-Golomb ruler.
\end{lem}

\begin{proof}
Classify the ordered pairs $(a,b)\in (\cV \cup \cW)^2$ with $a<b$ by the their coordinates:
\begin{description}[noitemsep]
    \item[type $VV$:] $a,b\in\cV$, and so $d=b-a< V_2 \le W_1-V_2$;
    \item[type $VW$:] $a\in\cV,b\in \cW$, and so $d=b-a > W_1-V_2$;
    \item[type $WW$:] $a,b\in\cW$, and so $d=b-a < m \le W_1-V_2$.
\end{description}
The differences $d\in \cV-\cV$ can arise from at most $g$ type $VV$ pairs, as $\cV$ is a $g$-Golomb ruler. Set $WW_d$ to be the left endpoint of each type $WW$ pair with difference $d$, and note that $|WW_d| \le g $ because $\cW$ is a $g$-Golomb ruler, and
    \[R \coloneqq \bigcup_{d\in \cV-\cV} WW_d, \qquad \left| R \right| \le g \cdot |(\cV-\cV)\cap\NN_{\ge 1}| \le g \binom{|\cV|}{2}.\]
Set
  \[\cW^{*}\coloneqq \cW\setminus R.\] 
Further, note now that each $d\in \cV-\cV$ has $d\not\in \cW^{*}-\cW^{*}$ and $d\not\in \cW^{*}-\cV$ because $d< V_2 \le W_1-V_2$. Thus, $d\in \cV-\cV$ has at most $g$ representations as a difference of elements in $\cV\cup \cW^{*}$.

Any difference $d$ that is at most $W_1-V_2$ can only arise from type $VV$ and type $WW$, and so can only arise from at most $g$ pairs of $\cV\cup\cW^{*}$. 

It remains to consider differences $d> W_1-V_2$. These can only arise from type $VW$ pairs. But if
\begin{equation*}
d=w - v = w' - v', \qquad w> w',v < v',\qquad w,w' \in \cW^{*},v,v' \in \cV,
\end{equation*}
then, $w - w' = v - v'$. But from the above construction, no positive difference in $\cV-\cV$ is in $\cW^{*}-\cW^{*}$. Thus, no such difference occurs more than once.
\end{proof}

\begin{proof}[Proof of Theorem~\ref{thm:limsup}]
Let $q\ge 3$ be a prime power with $q\equiv 1 \pmod{g }$. Set $q_1=q$ and $q_{i+1} = q_i^3$, so that each $q_i$ is a prime power with $q_i \equiv 1 \pmod{g }$. Set
\begin{equation*}
  m_i = \frac{q_i^2 - 1}{g }.
\end{equation*}
By Lemma~\ref{lem:finite lower}, we can choose a $g$-Golomb ruler $\cB_i \subseteq
[q_i+m_i, q_i+2m_i)$ with $|\cB_i| = q_i$ and, so that
\begin{equation*}
  \frac{|\cB_i|}{\sqrt{m_i}}
    = \frac{q_i}{\sqrt{(q_i^2 - 1)/g }}
    \longrightarrow \sqrt{g }.
\end{equation*}

We build $\cG$ as an increasing union $\cG \coloneqq \bigcup_i \cG_i$, with $\cG_i$ to be defined inductively.
Set $\cG_1 = \cB_1 \subseteq [q_1+m_1, q_1+2m_1)$. For $i \ge 1$,
suppose $\cG_i$ has been constructed and is a $g$-Golomb ruler
contained in $[0, q_i+ 2m_i)$. Put
\begin{align*}
  \cV &= \cG_i &\subseteq &[0,q_i+ 2m_i), \\
  \cW &= \cB_{i+1} &\subseteq &[q_{i+1}+m_{i+1}, q_{i+1}+2m_{i+1}).
\end{align*}
With 
  \[ V_2 \coloneqq q_i+2m_i,\quad  W_1 \coloneqq q_{i+1}+m_{i+1}, \quad m \coloneqq m_{i+1},\]
The inequalities of Lemma~\ref{lem:splice} hold as 
\begin{align*}
  W_1 - V_2 &=q_{i+1}+m_{i+1}-q_i-2m_{i} =q_i^3+ \frac{q_i^6-1}{g }-q_i -2\frac{q_i^2-1}{g }\\
  V_2 &\coloneqq q_i+2m_i = q_i + 2\frac{q_i^2-1}{g }\\
  m &\coloneqq m_{i+1} = \frac{q_i^6-1}{g },
\end{align*}
and so (using $q_i \ge q \ge 3$ and $g \ge 1$)
  \begin{align*}
     (W_1-V_2)-(V_2) 
     &= q_i^3+ \frac{q_i^6-1}{g }-2q_i -4\frac{q_i^2-1}{g } \\
     &\ge q_i^3+ q_i^6-1-2q_i -4q_i^2-4 \\
     &> 0
  \end{align*}
and
  \begin{align*}
    (W_1-V_2)-(m)
    &=q_i^3+ \frac{q_i^6-1}{g }-q_i -2\frac{q_i^2-1}{g }-\frac{q_i^6-1}{g } \\
    &=  q_i^3-q_i -2\frac{q_i^2-1}{g } \\
    &\ge  q_i^3-q_i -2(q_i^2-1) \\
    &=(q_i-2)(q_i-1)(q_i+1) \\
    &>0.
  \end{align*}
Apply Lemma~\ref{lem:splice} to obtain $\cW^{*} \subseteq \cW$ with
$\cV \cup \cW^{*}$ a $g$-Golomb ruler, and set
\begin{equation*}
  \cG_{i+1} \coloneqq \cG_i \cup \cW^{*} 
    \subseteq [0,\, q_{i+1}+2m_{i+1}).
\end{equation*}
Each $\cG_{i+1}$ is a $g$-Golomb ruler, and the union
  \[\cG \coloneqq \bigcup_{i=1}^\infty \cG_i\] 
is a $g$-Golomb ruler because any violating configuration involves finitely many elements and hence lies in some $\cG_{i+1}$.

We now consider the size of $\cG_{i+1}$. Clearly $\cG_{i+1} = \cG_i \cup W^{*}$, so that 
  \begin{align*}
    |\cG_{i+1} | 
      &= |\cG_i| + |\cB_{i+1}| - |\cB_{i+1} \setminus \cW^{*}| \\
      &\ge q_{i+1} -g \binom{|\cV|}{2} \\
      &= q_i^3 - g \binom{|\cG_i|}{2}.
  \end{align*}
Since $\cG_i$ is a $g$-Golomb ruler contained in $[0,q_i+2m_i)$, we know that $|\cG_i| \le 3\sqrt{q_i+2m_i}$, so that
  \[g \binom{|\cG_i|}{2} =O( q_i^2).\]
Thus, \( |\cG_{i+1}| \ge q_i^3- O(q_i^2). \)

We are now ready to finish the proof:
\begin{align*}
  \limsup_{n\to \infty} \frac{|\cG \cap [0,n)|}{\sqrt{n}}
  &\ge \lim_{i\to \infty} \frac{|\cG \cap [0,q_{i+1}+2m_{i+1})|}{\sqrt{q_{i+1}+2m_{i+1}}} \\
  &= \lim_{i\to\infty} \frac{ |\cG_{i+1}| }{\sqrt{q_i^3+2(q_i^6-1)/g }} \\
  &= \lim_{i\to\infty} \frac{q_i^3 - O(q_i^2)}{\sqrt{2/g } \, q_i^3+O(1)} \\
  &= \frac{\sqrt g }{\sqrt 2}. \qedhere
\end{align*}
\end{proof}

\subsection{Nonrigorous thoughts about the proof}
Lemma~\ref{lem:splice} is only used in the circumstance where $\cW$ is a $g$-Golomb ruler modulo $m$ (not merely a $g$-Golomb ruler in the integers). We thus have the option of translating and dilating $\cW$ modulo $m$ before projecting into $[W_1,W_1+m)$. The $2$ in the ``$\sqrt{2}$'' in the statement of Theorem~\ref{thm:limsup} is the same as the $2$ in our projection of $\cW$ into (essentially) $[m,2m)$. The author has expended considerable effort trying to use rotations of $\cW$ to allow a projection into $[0.99m,1.99m)$, and is now convinced that this does not work. The essential obstruction for $g =1$ is that dense Sidon sets are uniformly distributed, so that every rotation of $\cW$ has about $1\%$ of its elements in the interval $[0.99m,m)$, and this interacts with $\cV$ and $\cW\cap[1.97m,1.99m)$ enough to ruin the approach. The author has not considered dilations carefully, but they have the same fundamental obstruction described above.

\section{Further problems}\label{sec:problems}

The first problem suggested by this work is improving the constants in Theorem~\ref{thm:liminf} and~\ref{thm:limsup}, and the second problem is to provide bounds on how much the constants can be improved.

Suppose that $\cA$ is a Sidon set contained in $[0,NM)$ with $|\cA|\approx \sqrt{MN}$. What are the possible values of
  \[ \min \left\{\frac{A(N)}{f(N)}, \ldots, \frac{A(\ell N)}{f(\ell N)},\dots, \frac{A(MN)}{f(MN)}\right\}\]
for various functions $f$ and parameters $M$? Fine-distribution results of this nature would be helpful in some applications.

We are aware of no infinite construction of $g$-Golomb rulers other than the greedy construction. Can Ruzsa's construction~\cite{1998.Ruzsa-a} of a ``dense'' infinite $1$-Golomb ruler be extended to $g>1$?

\section*{Tool and computational resource disclosure}
This work was developed in interaction with Anthropic's \emph{ClaudeAI}, which was helpful in some ways and an incredible time-sink in others. Algebra was checked with Wolfram's \emph{Mathematica 14.3}. Lamport's \LaTeX\ was used both for typesetting and interacting with \emph{Claude}. Harmonic's \emph{AristotleAI} located a half-dozen typos and small errors (all now removed). Finally, as this work is Part I of a series of 3 papers, a substantial refactoring across the 3 works was need, and was conducted entirely by the human involved.

\end{document}